\documentclass[a4paper]{article}
\usepackage{amsthm,amssymb,amsmath,enumerate,graphicx,epsf,mathabx}
\pdfoutput=1
\newcommand{\COLORON}{0}
\newcommand{\NOTESON}{0}
\newcommand{\Debug}{0}
\usepackage{bbold}
\usepackage[usenames]{color} 
\usepackage{amsthm,amssymb,amsmath,bbm,enumerate,graphicx,epsf,stmaryrd}
\usepackage[bookmarks, colorlinks=false, breaklinks=true]{hyperref} 

\usepackage{authblk}

\newcommand{\comment}[1]{}
\newcommand{\COMMENT}[1]{}

\definecolor{darkgray}{rgb}{0.3,0.3,0.3}
\newcommand{\defi}[1]{{\color{darkgray}\emph{#1}}}

\newcommand{\acknowledgement}{\section*{Acknowledgement}}

\comment{
	\begin{lemma}\label{}	
\end{lemma}
\begin{proof}

\end{proof}

\begin{theorem}\label{}
\end{theorem} 
\begin{proof} 	

\end{proof}

}

\newtheorem{proposition}{Proposition}[section]

\newtheorem{theorem}[proposition]{Theorem}
\newtheorem{corollary}[proposition]{Corollary}

\newtheorem{lemma}[proposition]{Lemma}

\newtheorem{conjecture}{{Conjecture}}[section]

\newtheorem{problem}[conjecture]{{Problem}}

\newtheorem{examp}[proposition]{Example}

\newcommand{\FIG}{0}

\ifnum \NOTESON = 1 \newcommand{\note}[1]{ 

\hspace*{-30pt}
	{\color{blue}  NOTE: \color{Turquoise}{\small  \tt \begin{minipage}[c]{1.1\textwidth}  #1 \end{minipage} \ignorespacesafterend }} 
	
	}
\else \newcommand{\note}[1]{} \fi

\newcommand{\afsubm}[1]{ \ifnum \Debug = 1 {\mymargin{#1}}
\fi} 

\ifnum \Debug = 1 
\else  \fi

\ifnum \FIG = 1 \newcommand{\fig}[1]{Figure ``{#1}''}
\else \newcommand{\fig}[1]{Figure~\ref{#1}} \fi

\ifnum \FIG = 1 
\else  \fi

\ifnum \Debug = 1 \usepackage[notref,notcite]{showkeys}
\fi

\ifnum \COLORON = 0 \renewcommand{\color}[1]{}
\fi

\newcommand{\N}{\ensuremath{\mathbb N}}
\newcommand{\R}{\ensuremath{\mathbb R}}

\newcommand{\Z}{\ensuremath{\mathbb Z}}

\newcommand{\BS}{\ensuremath{\mathbb S}}

\newcommand{\cc}{\ensuremath{\mathcal C}}

\newcommand{\ce}{\ensuremath{\mathcal E}}

\newcommand{\ct}{\ensuremath{\mathcal T}}
\newcommand{\cu}{\ensuremath{\mathcal U}}

\newcommand{\sm}{\backslash}

\makeatletter
\DeclareRobustCommand{\cev}[1]{%
  \mathpalette\do@cev{#1}%
}
\newcommand{\do@cev}[2]{%
  \fix@cev{#1}{+}%
  \reflectbox{$\m@th#1\vec{\reflectbox{$\fix@cev{#1}{-}\m@th#1#2\fix@cev{#1}{+}$}}$}%
  \fix@cev{#1}{-}%
}
\newcommand{\fix@cev}[2]{%
  \ifx#1\displaystyle
    \mkern#23mu
  \else
    \ifx#1\textstyle
      \mkern#23mu
    \else
      \ifx#1\scriptstyle
        \mkern#22mu
      \else
        \mkern#22mu
      \fi
    \fi
  \fi
}

\makeatother

\newcommand{\pth}[2]{\ensuremath{#1}\text{--}\ensuremath{#2}~path}

\newcommand{\arc}[2]{\ensuremath{#1}\text{--}\ensuremath{#2}~arc}

\newcommand{\g}{\ensuremath{G\ }}
\newcommand{\G}{\ensuremath{G}}

\newcommand{\floor}[1]{\ensuremath{\left\lfloor #1 \right\rfloor}}

\newcommand{\Ex}{\mathbb E}

\newcommand{\Cg}{Cayley graph}

\newcommand{\Lr}[1]{Lemma~\ref{#1}}

\newcommand{\Tr}[1]{Theorem~\ref{#1}}
\newcommand{\Trs}[1]{Theorems~\ref{#1}}
\newcommand{\Sr}[1]{Section~\ref{#1}}

\newcommand{\Prb}[1]{Problem~\ref{#1}}
\newcommand{\Cr}[1]{Corollary~\ref{#1}}

\newcommand{\Dr}[1]{De\-fi\-nition~\ref{#1}}

\renewcommand{\iff}{if and only if}
\newcommand{\fe}{for every}

\newcommand{\st}{such that}

\newcommand{\ti}{there is}

\newcommand{\wrt}{with respect to}

\newcommand{\labtequ}[2]{
 \begin{equation} \label{#1} 	\begin{minipage}[c]{0.9\textwidth}  #2 \end{minipage} \ignorespacesafterend \end{equation} }

\newcommand{\mymargin}[1]{
 \ifnum \Debug = 1
  \marginpar{%
    \begin{minipage}{\marginparwidth}\small%
      \begin{flushleft}%
        {\color{blue}#1}%
      \end{flushleft}%
   \end{minipage}%
  }%
 \fi
}%

\newcommand{\mySection}[2]{}

\usepackage{mathtools,enumitem}
\usepackage[percent]{overpic}

\usepackage{authblk}
\usepackage[greek,english]{babel}  
\usepackage{hyperref}
\RequirePackage{latexsym} \RequirePackage{amsthm}
\RequirePackage{amsmath}

\usepackage{enumerate}
\RequirePackage{amssymb} \RequirePackage{makeidx}
\usepackage{dsfont}
\usepackage{mathrsfs}

\newtheorem{Def}{Definition}

\newcommand{\LSSC}{large-scale-simply-connected}

\newcommand{\asdim}{\mathrm{asdim}}
\newcommand{\ANdim}{\mathrm{ANdim}}
\newcommand{\diam}{\mathrm{diam}}

\usepackage{authblk}

\begin{document}

\title{Triangulations of uniform subquadratic growth are quasi-trees}

\author[1]{Itai Benjamini}
\affil[1]{
{Weizmann Institute, Israel}
}
\author[2]{Agelos Georgakopoulos\thanks{Supported by the European Research Council (ERC) under the European Union's Horizon 2020 research and innovation programme (grant agreement No 639046), and by EPSRC grant EP/V048821/1.}}
\affil[2]{  {Mathematics Institute}\\
 {University of Warwick}\\
  {CV4 7AL, UK}}

\date{\today}
\maketitle

\begin{abstract}
It is known that for every $\alpha \geq 1$ there is a planar triangulation in which every ball of radius $r$ has size $\Theta(r^\alpha)$. We prove that for $\alpha <2$ every such triangulation is quasi-isometric to a tree. The result extends to Riemannian 2-manifolds of finite genus, and to large-scale-simply-connected graphs. We also prove that every planar triangulation of asymptotic dimension 1 is quasi-isometric to a tree.
\end{abstract}

{\bf{Keywords:} }planar triangulation, 2-manifold, uniform volume growth, quasi-tree, asymptotic dimension.\\

{\bf{MSC 2020 Classification:}} 05C10, 05C12, 57M15, 57M50, 51F30, 60G99.
\maketitle

\section{Introduction}
Motivated by an observation of physicists that in certain planar triangulations the size of the ball of radius $r$ is of order $r^4$ \cite{ADJ, AngGro}, it was proved in \cite{BeSchrRec} that for every $\alpha \geq 1$ there is a triangulation of the plane in which every metric ball $B_v(r)$ of radius $r$ has size $|B_v(r)|= \Theta(r^\alpha)$ independently of the choice of the  centre $v$ of the ball. The constructions of \cite{BeSchrRec} are quasi-isometric to trees. Our first result is that for $\alpha <2$, this is not a coincidence, i.e.\ every such triangulation must be quasi-isometric to a tree.
A \defi{planar triangulation} is a connected plane graph every edge of which is contained in two facial triangles ---see \Sr{def basic} for more detailed definitions.

\begin{theorem} \label{thm triang}
Let $G$ be a planar triangulation. Then either $G$ is quasi-isometric to a tree, or for every $r \in \N$ there is a vertex $v$ such that $|B_v(r)| > r^2$.
\end{theorem}

For every $\alpha > 2$ on the other hand, a construction of Ebrahimnejad \& Lee \cite{EbrLeePla} yields ---after minor modifications--- planar  triangulations that are not quasi-isometric to a tree and satisfy $|B_v(r)|= \Theta(r^\alpha)$ uniformly in $v$. Combining these constructions with \Tr{thm triang}, and recalling that the triangular lattice has quadratic growth rate, we deduce that there is a planar triangulation with uniform volume growth $\Theta(r^\alpha)$ which is not quasi-isometric to a tree \iff\ $\alpha \geq 2$. (Such triangulations must still have relatively small cutsets at all scales \cite{BePaGro}.)


The trees constructed in \cite{BeSchrRec} can easily be modified into  2-manifolds with uniform volume growth $\Theta(r^\alpha)$ for any $\alpha>1$. We will also prove the following continuous analogue of \Tr{thm triang}:

\begin{corollary} \label{cor manifold}
Let $M$ be a  connected, complete, Riemannian 2-manifold of finite genus, with uniform subquadratic volume (i.e.\ area) growth. Then $M$ is quasi-isometric to a tree.
\end{corollary}

\comment{
	Combining \Tr{thm triang} with results in percolation theory we deduce that every planar triangulation \g of uniform subquadratic volume growth has very poor isoperimetric properties: \g must have  arbitrarily large subgraphs $H$ with boundary $|\partial H|$ of size $o(log |H|)$ (\Cr{cor cutsets}). This is in stark contrast with the result \cite{BeSchrPin} saying that graphs of pinched exponential volume growth satisfy an infinite-dimensional isoperimetric inequality.
}

\medskip
Our next result is a variant of \Tr{thm triang} in which planarity is replaced by the requirement that \g be \defi{large-scale-simply-connected (LSSC)}.
The \defi{cycle space} $\cc(G)$ of a graph \g is its first simplicial homology group over the 2-element field $\Z_2$.  

\begin{theorem} \label{thm lssc}
Let $G$ be a graph, and suppose $\cc(G)$ is generated by cycles of length at most $k$. Then either $G$ is quasi-isometric to a tree, or for every $r \in \N$ there is a vertex $v$ such that $|B_v(r)| \geq c(k) r^2$.
\end{theorem}


Here the $c(k)$ are universal constants. We remark that \Tr{thm triang} follows from \Tr{thm lssc} when $G$ is finite or 1-ended, because the cycle space is generated by the facial triangles in this case, but requires arguments specific to the planar case when \g has more ends. (A planar triangulation can have any number of ends up to the cardinality of the continuum; the duals of the graphs in \cite{cayley3} provide some examples. Likewise, $M$ need not be of finite type in \Cr{cor manifold}, since we are not imposing a restriction on the number of punctures; for example $M$ could be homeomorphic to the Cantor sphere.)
The LSSC condition cannot be relaxed in \Tr{thm lssc}, even if we impose a strong additional condition like planarity, as we show with an example in \Sr{sec graphs}.

\medskip
Fujiwara \& Whyte proved that every LSSC graph \g of asymptotic dimension 1 is a quasi-tree \cite{FujWhyNot}\footnote{For \Cg s this was also proved in \cite{GenAsy}.}. (The converse is well-known: every unbounded quasi-tree has asymptotic dimension 1.) We prove that if \g is a planar triangulation, then the LSSC condition can be dropped:

\begin{theorem} \label{asdim triang}
Let $G$ be an infinite planar triangulation. Then either\\ $\asdim(G)=2$, or $G$ is quasi-isometric to a tree (in which case $\asdim(G)=1$).
\end{theorem}

This uses a recent result that every planar graph has asymptotic dimension at most 2 \cite{BBEGLPS,JorLanGeo}.  \Tr{asdim triang} cannot be extended to planar graphs with arbitrarily long facial cycles as shown by \cite[Example 2.4.]{FujWhyNot}. We can replace $\asdim$ by the Assouad-Nagata dimension $\ANdim$ throughout the above discussion. In particular, we deduce that $\asdim(G)=\ANdim(G)$ when \g is a planar triangulation. We present an example in \Sr{sec asdim} showing that this equality does not hold for all planar graphs.

\medskip
It is a consequence of Gromov's theorem \cite{GroGro} that there is no group of volume growth of order $\Theta(r^\alpha)$ for non-integer $\alpha$. De la Harpe \cite{dlHarpe} asks for an elementary proof of this fact. \Tr{thm lssc} easily implies that there is no finitely presented group of superlinear but subquadratic growth, see \Cr{cor VT}. Alternative elementary proofs of this fact, covering the infinitely presented case as well, are provided in \cite{ImrSeiBou,JusGro,WilDriEff}. Some elementary proofs of the fact that every group of linear growth is quasi-isometric to $\Z$ have been provided by Bill Thurston\footnote{https://mathoverflow.net/questions/21578/is-there-a-simple-proof-that-a-group-of-linear-growth-is-quasi-isometric-to-z.}.

\section{Preliminaries}

\subsection{Basic definitions} \label{def basic}

We use standard graph-theoretic terminology following e.g.\ \cite{DiestelBook05}.
A \defi{plane graph} is a subset of $\R^2$ that is homeomorphic to a graph when the latter is viewed as an 1-complex. In other words, a plane graph is a planar graph endowed with a fixed embedding in $\R^2$. A \defi{face} of a plane graph \g is a component of $\R^2 \sm G$. A \defi{facial triangle} of \g is a cycle consisting of three edges that bounds a face.
A \defi{planar triangulation} is a connected plane graph every edge of which is contained in two facial triangles. For example, the 1-skeleton of every triangulation of $\R^2$ in the topological sense is a planar triangulation.

The \defi{edge space} $\ce(G)$ of a graph $G=(V,E)$ is the vector space $\mathbb{F}_2^E$ over the 2-element field $\mathbb{F}_2$, where vector addition amounts to symmetric difference. The \defi{cycle space} $\cc(G)$ is the subspace of $\ce(G)$ generated by the edge-sets of cycles.

The \defi{graph distance} $d(x,y)$ between two vertices $x,y\in V$ is the minimum number of edges in an \pth{x}{y}\ in \G. 

\subsection{Manning's theorem} \label{Manning}

A \defi{quasi-isometry} between graphs $G=(V,E)$ and $H=(V',E')$ is a map\\ $f: V \to V'$ \st\ the following hold for fixed constants $M\geq 1, A\geq 0$:
\begin{enumerate}
\item $M^{-1} d(x,y) -A \leq d(f(x),f(y))\leq M d(x,y) +A $ \fe\ $x,y \in V$, and
\item \fe\ $z\in V'$ \ti\ $x\in V$ \st\ $d(z,f(x))\leq A$.
\end{enumerate}
Here $d(\cdot,\cdot)$ stands for the graph distance in the corresponding graph $G$ or $H$. We say that $G$ and $H$ are \defi{quasi-isometric}, if such a map $f$ exists. Quasi-isometries between arbitrary metric spaces are defined analogously.

\medskip
We now recall Manning's \cite{Manning} characterization of the graphs allowing a quasi-isometry to a tree, which graphs we will call \defi{quasi-trees}:

\begin{Def} \label{def BP}
A graph  $G=(V,E)$ has the $\delta$-\defi{Bottleneck Property (BP)},  if for all $x, y \in V$  there is a `midpoint' $m = m(x,                                                                                                                                                                                                                                                                           y)$ with $d(x,m) = d(y,m) = d(x,y)/2$, such that any path from $x$ to $y$ meets the ball  $B_m(\delta)$. (We allow $m$ to lie in the middle of an edge.)
\end{Def}

\begin{theorem}[\cite{Manning}] \label{Manning thm}
A graph $G$ is quasi-isometric to a tree \iff\ it satisfies the $\delta$-(BP) for some $\delta>0$.
\end{theorem}

This theorem generalises to arbitrary \defi{length spaces}, which will allow us to apply it to manifolds. A metric space $(X,d)$ is a \defi{length space}, if for every $x,y\in X$ and $\epsilon>0$ there is an \arc{x}{y}\ in $X$ of length less that $d(x,y)+ \epsilon$. The definition of the $\delta$-(BP) for a length space is similar, we just let $m$ be an approximate midpoint.

\begin{theorem}
\label{Manning LS}
A length space $(X,d)$ is quasi-isometric to a tree \iff\ it satisfies the $\delta$-(BP) for some $\delta>0$.
\end{theorem}

\section{Results for Triangulations and Graphs} \label{sec graphs}

For a graph $G$ and a subgraph $H\subseteq G$, we define the boundary\\ $\partial H:= \{v\in V(H) \mid \text{ there is } vw\in E(G) \text{ with } w\not\in V(H) \}$. The following lemma will be used in the proof of \Tr{thm triang}. 

\begin{lemma} \label{lem triang}
Let $G$ be a planar triangulation, and let $H$ be a finite connected subgraph of $G$. Suppose two vertices $x,y \in V(H)$ are connected by a path $P$ in $(G \sm H) \cup \{x,y\}$. Then $x,y$ are connected by a path in $\partial H$. 
\end{lemma}
\begin{proof}

Since $H$ is connected, there is an $x$--$y$~path $Q$ in $H$. Then $C:= P \cup Q$ is a cycle (\fig{Fig P Q}). Pick one of the two sides $A$ of $C$ ---i.e.\ one of the two components into which $C$ separates the plane by the Jordan curve theorem--- and let $\ct$ denote the set of all facial triangles of $G$ lying in $A\cup C$ and having all their three vertices in $H$. Let $K$ be the element of $\cc(G)$ defined by the sum $E(C) + \bigoplus_{T\in \ct} E(T)$, where $E(T)$ denotes the edge-set of $T$. (In the example of \fig{Fig P Q}, $K$ happens to be a cycle $P\cup P'$, but in general $K$ will be more complicated if $H$ has `holes'.) Notice that $E(P) \subset K$. Since every element of the cycle space of a finite graph can be written as a disjoint union of edge-sets of cycles  \cite[Proposition~1.9.2.]{DiestelBook05}, and no internal vertex of $P$ is incident with an element of $\ct$, there is a cycle $C'$ containing $P$ such that $E(C') \subset K$. We claim that $P':= C' \sm P \subseteq \partial H$. Since $P'$ is an $x$--$y$~path by definition, this claim proves our statement.

\begin{figure}[htbp]
\vspace*{-5mm}
\centering
\noindent
\begin{overpic}[width=.6\linewidth]{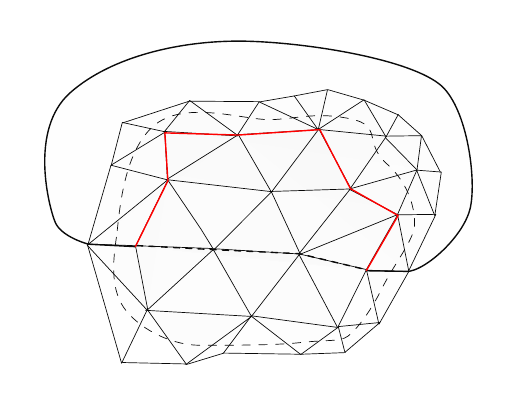}
\put(24,34){$x$}
\put(68,30){$y$}
\put(25.5,30.8){$\bullet$}
\put(70,26.4){$\bullet$}
\put(12,46){$A$}
\put(44,64){$P$}
\put(46,33){$Q$}
\put(49,49){$P'$}
\end{overpic}
\begin{minipage}[c]{0,95\textwidth}
\vspace*{-5mm}
\caption{The situation in the proof of \Lr{lem triang}. The subgraph $H$ is enclosed by the dashed curve, and $\ct$ comprises those triangles enclosed between the bold dashed path $Q$ and the path $P'$ depicted in red, if colour is shown.}
 \label{Fig P Q}
\end{minipage}
\end{figure}
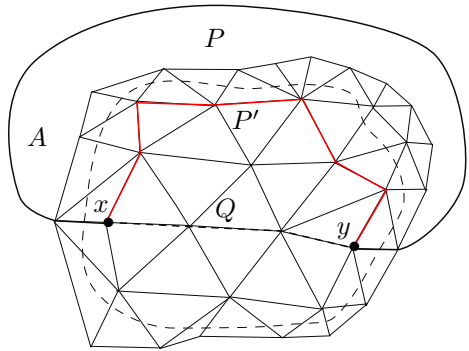

To establish the claim, we will show the stronger $V(K) \sm P \subseteq \partial H$, where $V(K)$ denotes the set of vertices incident with an edge of $K$. To see this, let $e=uv$ be an edge in $K$, and let $T_1,T_2$ be the two facial triangles containing $e$. Notice that $u,v\in V(H)$ by the definitions, so it remains to check that at least one other vertex of $T_1 \cup T_2$ is not in $H$. 

There are two cases. If $e\in E(C)$, then it cannot be that both $T_1,T_2$ are in \ct, and therefore  none of them is in \ct\ since $e\in K$. In this case the vertex $w$ of $T_1 \cup T_2$ inside $A$ is not in $H$ since  $uvw$ is not in \ct. The other case is where $e\not\in E(C)$, and therefore $e$ lies in $A$. Then exactly one of $T_1,T_2$ must be in \ct\ for $e$ to be in $K$, which means that the other $T_i$ contains a vertex not in $H$ as desired.
\end{proof}

We now prove our first main theorem, by repeatedly applying the above lemma to the `bottlenecks' of \Dr{def BP}:

\begin{proof}[Proof of \Tr{thm triang}]
Suppose $G$ is not quasi-isometric to a tree. Given $r \in \N$, Manning's \Tr{Manning thm} says that $G$ does not have the $r$-(BP). Thus we can find two vertices $p,q \in V(G)$ and a $p$--$q$~geodesic $\Gamma$, such that letting $m$ be the  midpoint of $\Gamma$, the ball $B_m(r)$ does not separate $p$ from $q$. In particular, $p,q$ lie outside $B_m(r)$. Let $P$ be a $p$--$q$~path outside $B_m(r)$. 

For every $i\in [1,r]$, we claim that there is a path $P_i$ in $\partial B_m(i)$ joining the two points $p_i,q_i$ of $\Gamma \cap \partial B_m(i)$. Indeed, $(P\cup \Gamma) \sm B_m(i-1)$ contains a $p_i$--$q_i$~path in $(G\sm B_m(i)) \cup \{p_i,q_i\}$, and so our claim follows from \Lr{lem triang}, applied with $H=B_m(i)$. 

Notice that the $P_i, i\in [1,r]$ are pairwise disjoint, they are contained in $B_m(r)$, and $|P_i| \geq d(p_i,q_i)=2i$. Thus $|B_m(r)| \geq \sum_{i\in [1,r]} 2i >r^2$.
\end{proof}

\Tr{thm lssc} can be proved along the same lines, except that instead of \Lr{lem triang} we use the following observation of Timar \cite{TimCut}. We say that $G$ is \defi{$k$-SC} ($k$-simply-connected), if $\cc(G)$ is generated by a set of cycles each of length at most $k$.

\begin{lemma}[{\cite[Theorem 5.1.]{TimCut}}] \label{lem Timar}
Let $G$ be a $k$-SC graph, and let $H$ be a finite connected subgraph of $G$. Suppose two vertices $x,y \in V(H)$ are connected by a path $P$ in $(G \sm H) \cup \{x,y\}$. Then $x,y$ are connected by a path each vertex of which is at distance at most $k/2$ from $\partial H$. 
\end{lemma}
(Timar's formulation is different but equivalent: it says that if $\Pi$ is a minimal cut separating two vertices $u,v$ of \G, and $\Pi_1,\Pi_2$ is a proper bipartition of $\Pi$, then there are vertices $x_i\in \Pi_i$ with $d(x_1,x_2)\leq k/2$. To deduce \Lr{lem Timar} from this formulation, let $u$ be any vertex of $H$, let $v$ be any vertex of $P$, and let $\Pi$ be the set of edges of $G$ between $H$ and the component of $G\sm H$ meeting $P$. The above can be rephrased as saying that any two edges in $\Pi$ can be connected by a sequence of paths, each starting and ending in $\Pi$, and each of length at most $k/2$.)

\begin{proof}[Proof of \Tr{thm lssc}]
Given $r \in \N$, define $p,q,m$ as in the proof of \Tr{thm triang}. For every $i\in [1,r]$, we can apply \Lr{lem Timar} with $H=B_m(i)$ to obtain a path $P_i$ at distance at most $k/2$ from $\partial B_m(i)$ joining the two points $p_i,q_i$ of $P \cap \partial B_m(i)$. Let $P'_j= P_{(k+1)j}, j\in [1,r/(k+2)]$, and notice that the $P'_j$ are pairwise disjoint, they are contained in $B_m(r)$, and $|P'_j| \geq d(p_{(k+1)j},q_{(k+1)j})=2(k+1)j$. Thus $|B_m(r)| \geq \sum_{j\in [1,r/(k+2)]} 2(k+1)j > (k+1) \frac{r^2}{(k+2)^2}$.
\end{proof}

We finish this section with an example showing that the condition of large-scale-simple-connectedness cannot be relaxed in \Tr{thm lssc}. Let $T$ be a tree of uniform volume growth $\Theta(r^{1.5})$, as provided in \cite{BeSchrRec}. Let $T'$ be a copy of $T$. Choose an infinite sequence $\{v_i\}_{i\in \N}$ of leaves of $T$, and identify each $v_i$ with its copy $v'_i$ in $T'$ to obtain a planar graph \G. Notice that the distance between $v_i$ and $v_j$ is the same in each of the three graphs $T, T'$ and \G. It follows that for every $r\in \N$ and $v\in V(T)$, we have $B^G_v(r) \subseteq B^T_v(r) \cup B^{T'}_{v'}(r)$
, i.e.\ a ball of a given radius in \g is at most twice as large as a ball of the same radius in $T$. By choosing the $v_i$ appropriately, e.g.\ so that $d(v_i,v_j)\geq 2^i$ \fe\ $j<i$, we can ensure that \g is not quasi-isometric to a tree.

\subsection{The asymptotic dimension of planar triangulations} \label{sec asdim}

We now prove  \Tr{asdim triang} using some of the ideas of the proof of \Tr{thm triang}.

We recall one of the standard definitions of asymptotic dimension $\asdim(X)$ of a metric space $(X,d)$  from e.g.\ \cite{JorLanGeo}. The reader can think of $X$ as being a graph endowed with its graph distance $d$. We include this definition for the purpose of fixing notation, the reader can consult other sources to gain more intuition.
 
Let \cu\ be a collection of subsets of $(X,d)$ ---usually a cover. For $s > 0$, we say that \cu\ is \defi{$s$-disjoint}, if
$d(U,U') := inf\ \{d(x, x') \mid x \in U, x' \in U'\} \geq s$
whenever $U,U'\in \cu$ are distinct. More generally, we say that \cu\ is \defi{$(n+1,s)$-disjoint}, if $\cu = \bigcup_{i=1}^{n+1} \cu_i$ for subcollections $\cu_i$ each of which is $s$-disjoint. The indices $1,\ldots,n+1$ are the \defi{colours} of \cu. We define the \defi{asymptotic dimension} $\asdim(X)$ as the smallest $n$ \st\ \fe\ $s\in \R_+$ \ti\ an $(n+1,s)$-disjoint cover \cu\ of $X$ with $\sup_{U\in \cu} \diam(U)< \infty$. Here, the diameter $\diam(U)$ is measured \wrt\ the metric $d$ of $X$. We say that \cu\ is \defi{$D$-bounded} for any $D> \sup_{U\in \cu} \diam(U)$. 

The \defi{Assouad-Nagata dimension} $\ANdim(X)$ is defined by putting a stronger restriction on the diameters: we define $\ANdim(X)$  to be the smallest $n$ for which there is $c\in \R$ \st\ \fe\ $s\in \R_+$ \ti\ an $(n+1,s)$-disjoint and $cs$-bounded cover \cu\ of $X$.

\begin{proof}[Proof of \Tr{asdim triang}]
It has been proved that every planar graph has asymptotic dimension at most 2 \cite{BBEGLPS,JorLanGeo}. Thus it only remains to show that $\asdim(G)=1$ implies that $G$ is quasi-isometric to a tree when \g is a planar triangulation.  We will prove the following stronger statement, where  $d$ denotes the graph distance of \G:
\labtequ{colours}{If \g admits a 2-colouring $c: V(G) \to \{0,1\}$ \st\ for some $r\in \N$, every monochromatic connected subgraph of \g has diameter less than $r$ \wrt\ $d$, then \g is a quasi-tree.}
To see that \eqref{colours} is satisfied when $\asdim(G)=1$, let \cu\ be a (2,2)-disjoint and $r$-bounded cover of \G, consisting of subcollections $\cu_1,\cu_2$, and set $c(v)=i$ if $v\in \bigcup \cu_i$. As an exercise, the reader could try to check that the triangular lattice $T$ (i.e.\ the planar triangulation with all vertices having degree 6) does not admit a 2-colouring as in \eqref{colours}, and deduce that $\asdim(T)>1$.

\medskip
To prove \eqref{colours}, we assume such $c$ exists, and we claim that \g has the $R$-BP for $R:=10r$ ---we are being generous--- from which the result follows via Manning's \Tr{Manning thm}.

For if not, then as before we can find two vertices $p,q \in V(G)$ and a $p$--$q$~geodesic $\Gamma$, such that letting $m$ be the  midpoint of $\Gamma$, the ball $B_m(R)$ does not separate $p$ from $q$. Let $P$ be a $p$--$q$~path outside $B_m(R)$. 

A \defi{monochromatic component} is a maximal connected subgraph of \g on which $c$ is constant. Let $C_0$ be the monochromatic component of $m$ in $c$. 
By our assumption, $C_0$ has diameter less than $r$ \wrt\ $d$, and in particular it is contained in $B_m(r)$. 
The idea is to recursively apply \Lr{lem triang} starting from $C_0$, to obtain longer and longer monochromatic components (in alternating colours) surrounding it, contradicting our assumption that monochromatic connected subgraphs of \g have bounded diameters.

To make this precise, let $C_0, C_1, \ldots C_k$ be a longest sequence of  monochromatic components (in alternating colours) with all the following properties. The two \defi{parts of $\Gamma$} are the two components into which $m$ separates $\Gamma$. 
\begin{enumerate}
\item \label{h i} $C_i$ contains a path joining the two parts of $\Gamma$ for all $i>0$;
\item \label{h ii} $C_i$ separates $C_{i-1}$ from $P$ in the subgraph $\Gamma \cup P$ for all $i>0$;
\item \label{h iii} $C_i  \subseteq B_m(9r)$, and
\item \label{h iv} $C_i\cap \Gamma \subseteq B_m(r)$.
\end{enumerate}
Since $C_0$ satisfies these conditions by definition, such a sequence always exists, although it may comprise a single member $C_0=C_k$. Properties \ref{h i} and \ref{h ii} together with the finiteness of $\Gamma$ guarantee that the sequence terminates. 

An example figure can be obtained as follows. Draw a sequence of simple closed curves $S_0, S_1, \ldots S_k$ in the plane, nested inside each other, and all surrounding a point $m$. Triangulate the interior of $S_0$, which we think of as $C_0$. Also triangulate each of the annuli bounded by $S_i$ and $S_{i-1}$, which we think of as $C_i$. Finally, draw a path $\Gamma$ through $m$, with both endpoints $p,q$ outside $S_k$. 

Back to our proof, let $C'_k$ be the set of neighbours of $C_k$. Notice that $C'_k$ is monochromatic, as all its vertices must have the opposite colour of $C_k$ since the latter was a  monochromatic component. Let $H:= C_k \cup C'_k$. Applying \Lr{lem triang}, which we can because $C_k$ satisfies \ref{h iii}, and so $H$ avoids $P$, we obtain a path $M \subset  \partial H \subseteq C'_k$ joining the two parts of $\Gamma$. Let $C_{k+1}$ be the monochromatic component  containing $M$,  which exists since $M \subseteq C'_k$ is monochromatic. Then $C_{k+1}$ satisfies \ref{h i} because it contains $M$, and it satisfies \ref{h ii} because $C'_k$ separates $C_{k}$ from $P$ in $\Gamma \cup P$. Notice that $d(m, M)\leq r+1$ because $C_k$ satisfies \ref{h iv} and $M \subseteq C'_k$. Therefore, if $C_{k+1} \supseteq M$ violates \ref{h iii}, then its diameter is at least $8r$ by the triangle inequality, contradicting our assumption.  If $C_{k+1}$ violates \ref{h iv}, then again this contradicts our assumption that $\diam(C_{k+1})<r$, because $C_{k+1}$ meets both parts of $\Gamma$ and $\Gamma$ is a geodesic through $m$.
\end{proof}

{\bf Remark:} Our proof is similar to that of Fujiwara \& Whyte, who proved that \LSSC\ graphs of asymptotic dimension 1 are quasi-trees \cite{FujWhyNot}. In fact our proof can be adapted to yield their result: If $G$ is $k$-SC, then its $k$th power $G^k$ is $3$-SC. We can apply the above proof to $G^k$, replacing \Lr{lem triang} by \Lr{lem Timar}. 

\medskip
It is a consequence of \Tr{asdim triang} that when \g is a planar triangulation, $\asdim(G)$ equals the Assouad-Nagata dimension $\ANdim(G)$. Indeed, it is well-known that $\ANdim(T)=1$ \fe\  tree $T$, and $\ANdim$ is invariant under quasi-isometry by definition. We now present an example showing that this property of planar triangulations does not extend to all planar graphs: we construct a planar graph \g with $\asdim(G)=1$ and $\ANdim(G)=2$. 

This graph is constructed as follows. For every $n\in \N$, let $H_n$ be the $n \times n$ square grid, i.e. the Cartesian product of the path of length $n$ with itself. Obtain $G_n$ from $H_n$ by subdividing each edge $n$ times. Let $G$ be the graph obtained from the union of all the $G_n$ after joining one vertex of $G_n$ with one vertex of $G_{n+1}$ with a path of length $n$ \fe\ $n\in \N$; it will not matter which vertex. 

We claim that $\ANdim(G)>1$. For if $\ANdim(G)=1$, then there is $c\in \R$ \st\ \fe\ $s\in \N$, there is an $(2,s)$-disjoint and $cs$-bounded cover of $G$ as defined in \Sr{sec asdim}. For every $n$, we then have a $(2,ns)$-disjoint and $cns$-bounded cover $\cu_n$ of $G$. The restriction of $\cu_n$ to $G_n$ yields a $(2,s)$-disjoint and $cs$-bounded cover $\cu'_n$ of $H_n$, because all distances in $G_n$ are $n$ times larger than corresponding distances in $H_n$. Using a standard compactness argument we can find a subsequence of the $\cu'_n$ converging to a $(2,s)$-disjoint and $cs$-bounded cover of the infinite square lattice $\Z^2$, contradicting the fact that $\ANdim(\Z^2)=2$. (A similar idea appears in \cite{BBEGLPS}.)

It remains to check that $\asdim(G)=1$. Given $s\in \R$, we can construct a $(2,s)$-disjoint cover $\cu= \cu_1 \cup \cu_2$ of \g as follows. Put all $G_n$'s with $n\leq 4s$ in one element $U$ of $\cu_1$. For the larger $G_n$'s, start with a proper 2-colouring of $H_n$ with colours $\{1,2\}$, and extend it to $G_n$ as follows. Colour each vertex $v$ of degree 2 with the colour of the nearest vertex $v'$ of degree 4 (inherited from the colouring of $H_n$) whenever $d(v,v')\leq s$. The yet uncoloured vertices form a family of paths of length at least $2s$ each. Easily, each such path $P$ can be decomposed into an even number of subpaths each of length between $s$ and $2s$. We colour these subpaths in alternating colours, so that each endvertex of $P$ receives the opposite colour of its neighbour outside $P$. We complete the definition of \cu\ by putting each monochromatic component of colour $i$ in $\cu_i$. It is straightforward to check that \cu\ is a $(2,s)$-disjoint cover of \G, and that $\sup_{U\in \cu} \diam(U)< \infty$. Thus  $\asdim(G)=1$ as claimed.

\comment{ 
Let $G$ be a planar triangulation of uniform subquadratic volume growth. By \Tr{thm triang} $G$ is quasi-isometric with a tree $T$, and since \g must have bounded degrees, it is easy to see that we can choose $T$ to be of bounded degree too. It is a well-known theorem of Lyons \cite{} that trees of polynomial growth have no percolation phase transition, i.e.\ their $p_c$ equals 1. It is also well known that the property $p_c = 1$ is   invariant under quasi-isometries between bounded degree graphs \cite{}. Thus we deduce $p_c(G)=1$. On the other hand, it is proved in \cite{analyticity}[Theorem 11.1.] that if \g satisfies a 

\begin{corollary} \label{cor cutsets}
Let $G$ be a planar triangulation \st\ $|B_v(r)|= O(r^\alpha)$ uniformly for all $v\in V(G)$ for $\alpha<2$.  Then \g has subgraphs $\{H_i\}_{i\in \N}$ \st\ $\lim \frac{|\partial H_i|}{\log |H_i|} = 0$.
\end{corollary}
}

\section{Manifolds}
In this section we prove continuous analogues of the above results on planar triangulations. We will state our results for Riemannian 2-manifolds, although our proofs apply more generally to any topological 2-manifold endowed with a metric that turns it into a length space, in particular to any Finsler 2-manifold. 
A 2-manifold is \defi{planar}, if it is homeomorphic to a subspace of the 2-sphere $\BS^2$. For a subset $X$ of a manifold we define $Area(X)$ to be its 2-dimensional Hausdorff measure $\mathcal{H}^2(X)$. (This definition generalises the Riemannian area.) As usual we denote by $B_p(r)$ the ball of radius $r$ centered at a point $p$. The main result of this section is

\begin{theorem} \label{thm surfaces}
Let $M$ be a connected, complete, planar, Riemannian 2-manifold. Then either $M$ is quasi-isometric to a tree, or for every $r \in \R_+$ there is $p \in M$ such that $Area(B_p(r)) > r^2/8$.
\end{theorem}

The proof of this follows the lines of the proof of \Tr{thm triang}. We will need the following lemma which is the analogue of \Lr{lem triang}, and will play a similar role in our proof. 

\begin{lemma} \label{lem con surf}
Let $S$ be a planar 2-manifold, and let $K\subseteq S$ be  a connected and compact subspace. Suppose $x,y \in \partial K$ are connected by an arc $P$ in $(S \sm K) \cup \{x,y\}$. Then \fe\ $\epsilon>0$, $x,y$ are connected by an arc contained in the $\epsilon$-neighbourhood of $\partial K$. 
\end{lemma}

{\bf Remark:} We have to consider a neighbourhood of $\partial K$ here, as we cannot find an \arc{x}{y}\ in $\partial K$ if $K$ is too fractal. For example, $K$ could be a pseudo-arc in $S= \R^2$, in which case $\partial K=K$ contains no non-trivial arc.
\begin{proof}[Proof of \Lr{lem con surf}]
It is well-known that $S$ admits a triangulation $T$ \cite[I.~46]{AhlSarRie}. Given $\epsilon>0$, we can subdivide the triangles of $T$ if necessary to obtain a triangulation $T_\epsilon$ of $S$ in which every triangle (i.e.\ 2-cell)  has diameter at most $\epsilon$ \wrt\ the metric of $S$. Let $G=G(\epsilon)$ denote the 1-skeleton of $T_\epsilon$, and note that $G$ is a planar triangulation by definition. 

Let $H\subset G$ denote the subgraph of $G$ consisting of the boundaries of the triangles of $T_\epsilon$ intersecting $K$. Since $K$ is connected, so is $H$, and since $K$ is compact, $H$ is finite. Let $x',y'$ be vertices of $\partial H$ in the boundary of triangles $\Delta_x, \Delta_y$ containing $x,y$ respectively, and let $P'$ be an \pth{x'}{y'} in $G$ at distance at most $\epsilon$ from $P$, which can be found inside the union of the triangles of $T_\epsilon$ intersecting $P$. Applying \Lr{lem triang} we obtain an \pth{x'}{y'} $Q_\epsilon$ in $\partial H$. Notice that $\partial H \cap K=\emptyset$, and $\partial H$ is contained in the $\epsilon$-neighbourhood of $K$. Thus $Q_\epsilon \subseteq \partial H$ is contained in the $\epsilon$-neighbourhood of $\partial K$. We extend $Q_\epsilon$ by an \arc{x}{x'} in $\Delta_x$ and a \arc{y}{y'} in  $ \Delta_y$ to obtain the desired \arc{x}{y}.
\end{proof}

In the proof of \Tr{thm triang} we used \Lr{lem triang} to find a sequence of paths $Q_i$ joining points at distance $2i$, and each contained in the boundary of a ball with a fixed center. We will argue analogously in the proof of \Tr{thm surfaces}, and the following lemma will be used to show that the union of the paths that \Lr{lem con surf} provides has large total area. 

\begin{lemma} \label{lem arcs}
Let  $\{P_t^\epsilon, t\in (0,T), \epsilon \in (0,1)\}$ be a family of arcs in a metric space $(X,d)$, and let $p_t^\epsilon, q_t^\epsilon$ denote the two endpoints of $P_t^\epsilon$. Suppose that for some $L>0$, and every $\epsilon \in (0,1)$, we have 
\begin{enumerate}
\item \label{i}  $d(p_t^\epsilon, q_t^\epsilon)\geq L$ \fe\ $t \in (0,T)$, and
\item \label{ii} $d(P_t^\epsilon,P_s^\epsilon) \geq |s-t| - \epsilon$ \fe\ $s,t \in (0,T)$.
\end{enumerate}
Then $\mathcal{H}^2(X) \geq  LT/4$. 
\end{lemma}

For example, $X$ could be a rectangle with side lengths $L,T$, and $P_t^\epsilon$ could be the horizontal arc at height $t$ perturbed within distance $\epsilon$.

\begin{proof}[Proof of \Lr{lem arcs}]

This is a rather straightforward application of the definition of Hausdorff measure. Recall that $\mathcal{H}^2(X)$ is defined as $\sup_{\delta \to 0} \mathcal{H}^2_\delta(X)$, where\\ $\mathcal{H}^2_\delta(X):= \inf \{ \sum \diam(U_i)^2 \}$, the infimum ranging over all countable covers $\{ U_i\}$ of $X$ satisfying $\diam(U_i)< \delta$ \fe\ $i$.

Given $\delta>0$, the idea is to choose a finite subfamily of our arcs $\{P_t^\epsilon\}$ whose pairwise distance is at least some increment $c$, and within each arc to choose a sequence of points obeying the same increment $c$, so that any set $U_i$ used to cover $X$ can only contain $O(\diam(U_i)^2/c^2)$ many of these points as $c\to 0$. Hereby we choose $\epsilon$ much smaller than $c$, e.g.\ $\epsilon= c^2$. 
As we have $\Omega({LT/c^2})$ points to cover in total, this provides the desired lower bound on $\mathcal{H}^2_\delta(X)$. 

To make this more precise, given $\delta<1/2$, choose $N>>1/\delta$,  let \\ $t_i:= i T \delta/N, i\in 1,2,\ldots \floor{\frac{N}{\delta}}$, and let $P_i:= P_{t_i}^{{(\delta/N)}^2}$. Notice that $d(P_i,P_j) \geq T (\delta/N |i-j| - ({\delta/N})^2)$ by \ref{ii}, and therefore any set $U$ with $\diam(U)< \delta$ meets at most $2\diam(U) N/\delta$ of the $P_i$'s. For each $i$, let $p_i^j, j\in 1,2, \ldots \floor{\frac{N}{\delta}}$ be a point on $P_i$ at distance $j L  \delta/N$ from a fixed endpoint of $P_i$, which exists by  \ref{i}. The triangle inequality implies that if $\diam(U)< \delta$ then $U$ meets at most $2 \diam(U) N/\delta$ of the $p_i^j$'s for any fixed $i$. Combined with the above remark, we deduce that $U$ meets at most $4 (\diam(U) N/\delta)^2$ points $p_i^j$ for any $i$ and $j$. Since we have at least $\frac{N^2}{\delta^2}$ points to cover in total, we deduce
$$\sum_{U \in \cu} 4 (\diam(U) N/\delta)^2 \geq \frac{N^2}{\delta^2}$$
for every cover $\cu$ of $X$ with diameters bounded by $\delta$, and so $\sum_{U \in \cu} (\diam(U))^2 \geq 1/4$. Thus $\mathcal{H}^2_\delta(X)\geq 1/4$.
\end{proof}

We are now ready to prove the main result of this section.
\begin{proof}[Proof of \Tr{thm surfaces}]
We follow the lines of the proof of \Tr{thm triang}. If $M$ is not quasi-isometric to a tree, then for every $r \in \R$, \Tr{Manning LS} provides two points $p,q \in M$ and a $p$--$q$~arc $\Gamma$ of length arbitrarily close to $d(p,q)$, such that letting $m$ be the  midpoint of $\Gamma$, the ball $B_m(r)$ does not separate $p$ from $q$. Let $P$ be a $p$--$q$~path outside $B_m(r)$. Applying \Lr{lem con surf} with $K=K(s):= B_m(s) $ for $s\in (r/2,r)$ we obtain a family  
$\{P_t^\epsilon, t\in (0,r/2), \epsilon \in (0,1)\}$ of arcs each joining two points of $\Gamma$ outside $B_m(r/2)$, which are thus at distance at least $r$ from each other. Here we used the standard fact that every closed bounded subspace of a complete, locally compact, length space is compact. Applying  \Lr{lem arcs} to this family we deduce $\mathcal{H}^2(B_m(r)) \geq r/8$. 
\end{proof}

We can now deduce \Cr{cor manifold} from \Tr{thm surfaces} as follows. We first perform a finite number of surgery operations along non-contractible, rectifiable, loops of $M$ to produce a 2-manifold $M'$ of genus 0 no ball of which has larger area than a ball of the same radius in $M$. To maintain completeness, we glue a disc of finite area along each boundary component we created, to obtain a complete planar surface $M''$. This has a significant influence on the volume growth for small radii only. \Tr{thm surfaces} thus yields that $M''$ is a quasi-tree, hence so is $M$ since it is quasi-isometric to $M''$. (We do not need to worry about losing smoothness when glueing in the aforementioned discs, because as mentioned at the beginning of this section, our proof of   \Tr{thm surfaces} only requires $M''$ to be a length space that is a topological 2-manifold, not necessarily a Riemannian one.)

\medskip
We remark that the completeness assumption cannot be dropped in \Tr{thm surfaces}: one can construct a planar manifold quasi-isometric to the graph obtained from an 1-way infinite path by replacing its $i$th edge with a cycle of length $i$. Such a planar graph is not a quasi-tree, and it has linear growth.

\section{Growth of groups} \label{sec groups}

As a corollary of \Tr{thm lssc}, we obtain an alternative proof that there is no finitely presented group of superlinear but subquadratic growth without using Gromov's theorem. For this we do not need to use Manning's theorem:

\begin{corollary} \label{cor VT}
There is no $k$-SC vertex-transitive graph \g of superlinear but subquadratic growth for any $k\in \N$.
\end{corollary}

\begin{proof}[(Sketch)]
If $G$ has the $\delta$-(BP) for some $\delta$, we can use a standard compactness argument to obtain a double-ray $R$ and a ball $S$ of radius $\delta$ that separates the two ends of $R$. It follows that $G$ has at least 2 ends, and so $G$ has either linear (if 2-ended \cite[Theorem~2.8 \& Proposition~3.1]{ImrSeiNot}) or exponential growth (if infinitely-ended \cite{HalWeg}). 

If $G$ does not have the $\delta$-(BP) then we apply \Tr{thm lssc} (omitting the first two sentences of the proof of \Tr{thm triang} that invoke \Tr{Manning thm}).
\end{proof}

\section{Questions}

As mentioned in the introduction, for every $\alpha >1$ there is an (1-ended) triangulation $T_\alpha$ of the plane $\R^2$ with uniform volume growth of order $\Theta(r^\alpha)$ \cite{BeSchrRec}. 
The cartesian product $T_\alpha \times \Z$ can be thought of as the 1-skeleton of a 3-complex, which we can triangulate to obtain a 3-dimensional simplicial complex  $C_\alpha$, in which each copy of a triangle of $T_\alpha$ appears in the boundary of two tetrahedra. Note that  $C_\alpha$ has uniform volume growth of order $\Theta(r^{\alpha+1})$, and it is homeomorphic to $\R^3$. We are interested in the structure of 3-dimensional simplicial complexes homeomorphic to $\R^3$ (or other 3-manifolds) that have uniform volume growth of order $\Theta(r^\beta)$ for $\beta<3$.  
The following question is an attempt to extend \Tr{thm triang} to three dimensions.

\begin{problem} \label{prob 3D}
Let $X$ be a simplicial complex homeomorphic to $\R^3$, with uniform volume growth $r^{\alpha}$ for some $\alpha <3$. Must $X$ be quasi-isometric with a contractible simplicial 2-complex? 
\end{problem}
The above construction $C_\alpha$ with $\alpha\in (1,2)$ provides some examples. Since $T_\alpha$ is quasi-isometric to a tree, $C_\alpha$ is  quasi-isometric to the cartesian product of a tree with the 2-way infinite path, which is contractible (and even collapsible, a stronger condition we could ask for in \Prb{prob 3D}; see e.g.\ \cite{AdiFunCat} for the definition).

\medskip
We remark that in \Trs{thm triang} and~\ref{thm lssc} the vertex $v$ has to depend on $r$, in other words, we cannot drop the uniformity of the  subquadratic growth if we want to obtain a quasi-tree. For example, let $T$ denote the triangular lattice in $\R^2$, and let $X= x_0 x_1 \ldots$ and $Y= y_0 y_1 \ldots$ be two geodesics emanating from a common vertex $x_0=y
_0$, such that $d(x_i,y_i)= \Theta(\sqrt{i})$. Cut $T$ along $X \cup Y$, throw away the larger piece, and identify $x_i$ with $y_i$ in the smaller piece. The resulting planar triangulation can be visualised as a parabolic cone. It has subquadratic volume growth, but not uniformly so, and it is not 
quasi-isometric with a tree. 

Still, for a unimodular random graph we can ask if a statement similar to   \Trs{thm triang} and~\ref{thm lssc}  holds under a subquadratic expected volume growth condition:

\begin{problem} \label{prob urg}
Let $(G,o)$ be a unimodular random graph\footnote{See \cite{unimodular} for definitions} which is $k$-SC for some $k$, such that $c_1 r< \Ex(|B_o(r)|)< c_2 r^2$ for some constants $c_1,c_2>0$.  Must $(G,o)$ be almost surely 1-ended? 
Must it have an infinite sequence of nested, bounded, cut sets  separating $o$ from infinity? Must every scaling limit of such a graph be a tree?
\end{problem}

Let $G$ be a planar triangulation of uniform subquadratic volume growth. By \Tr{thm triang} $G$ is quasi-isometric with a tree $T$, and since \g must have bounded degrees, it is easy to see that we can choose $T$ to be of bounded degree too. It is a well-known theorem of Lyons \cite{LyoRan} that trees of polynomial growth have no percolation phase transition, i.e.\ their $p_c$ equals 1. It is also well known that the property $p_c = 1$ is   invariant under quasi-isometries between bounded degree graphs \cite[Theorem 7.15]{LyonsBook}, and so is the exponent of the volume growth. Thus we deduce $p_c(G)=1$. This leads to the following question:

\begin{problem} \label{prob pc}
Is it true that every $(G,o)$ as in \Prb{prob urg} satisfies $p_c=1$ almost surely?
\end{problem}

Consider a unimodular random graph $(G,o)$ of super-linear but sub-quadratic expected volume growth. 
We do not impose the $k$-SC condition, so apart from the aforementioned quasi-trees, examples can be obtained from the sequence of graphs defining the Sierpinski gasket by sampling uniform random roots. In all the examples we know of,  simple random walk is  subdiffusive, see e.g.\ \cite[Proposition 8.11]{BarDif}. Is this the case for every such $(G,o)$?


\acknowledgement{We thank Panos Papazoglou and Federico Vigolo for helpful discussions. We thank Guy Lachman, Martin Winter and Geva Yashfe for helping us improve Problems~\ref{prob 3D} and \ref{prob urg}.}

\comment{
\section{ignore for now}
\begin{lemma} \label{lem triang}
Let $G$ be a $k$-SC graph of maximum degree $d$ which is not $(Q,Q)$--quasi-isometric to a tree. Then for every $r\in \N$, there is a ...subgraph $S\subseteq G$ isomorphic with the $r$-fold subdivision of the star $S_4$ with four leaves, such that 
$$d_G(x,y)> f_{k,d,Q}(d_S(x,y))$$
holds for every $x,y\in V(S)$, where $f_{k,d,Q}: \N \to \N$ is a diverging, universal function.
\end{lemma}
\begin{proof}
By \Tr{Manning} we can find two vertices $p,q \in V(G)$ and a $p$--$q$~geodesic $P$, such that letting $m$ be the  midpoint of $P$, the ball $B_m(r)$ does not separate $p$ from $q$. In particular, $p,q$ lie outside $B_m(r)$. Let $Q$ be a $p$--$q$~path outside $B_m(r)$... 
\end{proof}
}

\bibliographystyle{plain}
\bibliography{collective}

\begin{thebibliography}{10}

\bibitem{AdiFunCat}
K.~A. Adiprasito and L.~Funar.
\newblock {CAT(0) metrics on contractible manifolds}.
\newblock {\em arXiv:1512.06403}, 2015.

\bibitem{AhlSarRie}
Lars~V. Ahlfors and Leo Sario.
\newblock {\em Riemann {Surfaces}}.
\newblock Princeton University Press, 1960.

\bibitem{unimodular}
D.~J. Aldous and R.~Lyons.
\newblock Processes on unimodular random networks.
\newblock {\em Electronic Journal of Probability}, 12(54):1454--1508, 2007.

\bibitem{ADJ}
Jan Ambjorn, Bergfinnur Durhuus, and Thordur Jonsson.
\newblock {\em Quantum {Geometry}: {A} {Statistical} {Field} {Theory}
  {Approach}}.
\newblock Cambridge {Monographs} on {Mathematical} {Physics}. Cambridge
  University Press, 1997.

\bibitem{AngGro}
O.~Angel.
\newblock Growth and percolation on the uniform infinite planar triangulation.
\newblock {\em Geometric and Functional Analysis}, 13(5):935--974, 2003.

\bibitem{BarDif}
Martin~T. Barlow.
\newblock Diffusions on fractals.
\newblock In {\em Lectures on Probability Theory and Statistics}, number 1690
  in Lecture Notes in Mathematics, pages 1--121. Springer Berlin Heidelberg,
  1998.

\bibitem{BePaGro}
I.~Benjamini and P.~Papasoglu.
\newblock Growth and isoperimetric profile of planar graphs.
\newblock {\em Proceedings of the American Mathematical Society},
  139(11):4105--4111, 2011.

\bibitem{BeSchrRec}
I.~Benjamini and O.~Schramm.
\newblock Recurrence of distributional limits of finite planar graphs.
\newblock {\em Electronic Journal of Probability}, 6, 2001.

\bibitem{BBEGLPS}
M.~Bonamy, N.~Bousquet, L.~Esperet, C.~Groenland, C.-H. Liu, F.~Pirot, and
  A.~Scott.
\newblock Asymptotic {Dimension} of {Minor}-{Closed} {Families} and
  {Assouad}-{Nagata} {Dimension} of {Surfaces}.
\newblock 2020.
\newblock arXiv: 2012.02435.

\bibitem{dlHarpe}
Pierre de~la Harpe.
\newblock {\em Topics in {Geometric} {Group} {Theory}}.
\newblock The University of Chicago Press, 2000.

\bibitem{DiestelBook05}
Reinhard Diestel.
\newblock {\em Graph {T}heory \emph{(3rd edition)}}.
\newblock Springer-Verlag, 2005.
\newblock \\ Electronic edition available at:\\ {\small\tt
  http://www.math.uni-hamburg.de/home/diestel/books/graph.theory}.

\bibitem{EbrLeePla}
F.~Ebrahimnejad and J.~R. Lee.
\newblock On planar graphs of uniform polynomial growth.
\newblock {\em Probab.\ Theory Relat.\ Fields}, 180(3):955--984, 2021.

\bibitem{FujWhyNot}
K.~Fujiwara and K.~Whyte.
\newblock A note on spaces of asymptotic dimension one.
\newblock {\em Algebraic \& Geometric Topology}, 7(2):1063--1070, January 2007.

\bibitem{GenAsy}
T.~Gentimis.
\newblock Asymptotic dimension of finitely presented groups.
\newblock {\em Proc.\ Am.\ Math.\ Soc.}, 136(12):4103--4110, 2008.

\bibitem{cayley3}
A.~Georgakopoulos.
\newblock {The planar cubic Cayley graphs}.
\newblock {\em Memoirs of the {AMS}}, 250(1190), 2017.

\bibitem{GroGro}
M.~Gromov.
\newblock {Groups of polynomial growth and expanding maps. With an appendix by
  Jacques Tits}.
\newblock {\em Inst.\ Hautes \'Etudes Sci.\ Publ.\ Math.}, 53:53--73, 1981.

\bibitem{HalWeg}
R.~Halin.
\newblock {\"Uber unendliche Wege in Graphen}.
\newblock {\em Mathematische Annalen}, 157(2):125--137, 1964.

\bibitem{ImrSeiBou}
W.~Imrich and N.~Seifter.
\newblock A bound for groups of linear growth.
\newblock {\em Archiv der Mathematik}, 48(2):100--104, 1987.

\bibitem{ImrSeiNot}
W.~Imrich and N.~Seifter.
\newblock A note on the growth of transitive graphs.
\newblock {\em Discrete Mathematics}, 73(1):111--117, 1988.

\bibitem{JorLanGeo}
M.~J\o{}rgensen and U.~Lang.
\newblock {Geodesic spaces of low Nagata dimension}.
\newblock {\em {Annales Fennici Mathematici}}, 47:83--88, 2022.

\bibitem{JusGro}
J.~Justin.
\newblock Groupes et semi-groupes a croissance lineaire.
\newblock {\em C.~R.~Acad.~Sci.~Paris}, 273:212--214, 1971.

\bibitem{LyoRan}
R.~Lyons.
\newblock Random {Walks} and {Percolation} on {Trees}.
\newblock {\em The Annals of Probability}, 18(3):931--958, 1990.

\bibitem{LyonsBook}
Russell Lyons and Yuval Peres.
\newblock {\em Probability on Trees and Networks}.
\newblock Cambridge University Press, New York, 2016.
\newblock Available at {http://pages.iu.edu/~rdlyons/}.

\bibitem{Manning}
J.~F. Manning.
\newblock Geometry of pseudocharacters.
\newblock {\em Geometry \& Topology}, 9:1147--1185, 2005.

\bibitem{TimCut}
A.~Tim{\'a}r.
\newblock Cutsets in infinite graphs.
\newblock {\em Combinatorics, Probability and Computing}, 16:159--166, 2007.

\bibitem{WilDriEff}
A.~J. Wilkie and L.~van~den Dries.
\newblock An effective bound for groups of linear growth.
\newblock {\em Archiv der Mathematik}, 42(5):391--396, 1984.

\end{thebibliography}
\end{document}